\documentclass{amsart}
\usepackage{graphicx}
\usepackage{epstopdf}
\usepackage{pstricks}
\usepackage[mathscr]{eucal}

\theoremstyle{plain}
\newtheorem{prop}{Proposition}[section]
\newtheorem{coro}[prop]{Corollary}

\newtheorem{conj}[prop]{Conjecture}
\newtheorem{lemm}[prop]{Lemma}

\newtheorem{thm}[prop]{Theorem}

\theoremstyle{definition}

\def\mcg#1;#2{\Gamma_{#1,#2}}
\def\fg#1;#2{\Pi_{#1,#2}}
\def\tb#1;#2{\mathscr{K}_{\frac{#1}{#2}}}

\begin{document}

\title[ Polynomials of weaving links \& Whitney rank polynomials of the Lucas lattice ]
{Alexander and Jones   Polynomials of weaving 3-braid links and Whitney rank polynomials of Lucas lattice}

\author{Mark E. AlSukaiti}
\address{Department of Mathematical Sciences\\ College of Science\\ United Arab Emirates  University \\ 15551 Al Ain, U.A.E.}
\email{700035655@uaeu.ac.ae}

\author{Nafaa Chbili}
\address{Department of Mathematical Sciences\\ College of Science\\ United Arab Emirates  University \\ 15551 Al Ain, U.A.E.}
\email{nafaachbili@uaeu.ac.ae}
\urladdr{http://faculty.uaeu.ac.ae/nafaachbili}

\date{}

\begin{abstract}
We establish a relationship  between the  Jones  polynomial of generalized weaving knots of type $W(3,n,m)$ and the Chebyshev polynomial of the first kind. Consequently, we  prove
that the coefficients of the Jones polynomial of weaving knots are basically the
Whitney numbers of Lucas lattices.  Furthermore, we give  an explicit formula for the Alexander polynomial of  weaving knots $W(3,n)$  and we prove that  it satisfies Fox's trapezoidal conjecture.
\end{abstract}

\maketitle

\section{Introduction}

Let $f(x)=\sum_{k=0}^{n}\alpha_kx^k$ be a polynomial with real  positive coefficients. Then, $f$ is said to be logarithmically  concave,
or  log-concave for short,  if its coefficients satisfy the condition $\alpha_k^2\geq \alpha_{k-1}\alpha_{k+1}$ for all $0<k<n$. If there exists an integer $r \leq n/2$ such that:
$\alpha_0<\alpha_1< \dots < \alpha_{\frac{n}{2}-r}=\dots =\alpha_{\frac{n}{2}+r}> \dots ...>\alpha_{n-1}>\alpha_{n}$, then $f$ is said to be trapezoidal.  It is well known that every log-concave polynomial is trapezoidal. Polynomials and sequences with such properties have been subject to extensive studies.  The importance of these  polynomials in combinatorics,  algebra and geometry is highlighted in \cite{Hu}, \cite{Sta} and the references therein.\\

 An $n$-component link is an embedding of $n$ disjoint circles $\coprod S^{1}$ into the 3-dimensional sphere  $S^3$. A  knot is a one-component link.  A link diagram is a
regular planar  projection of the link together with hight information at each double point showing which of the two strands passes over the other.
  A link  is said to be alternating if it admits an alternating diagram; a diagram where
 the overpass and the underpass alternate as one travels along any strand of the diagram. \\
Let $n\geq 2$ be an integer  and $B_n$ be  the  group of braids on $n$ strands. This group is generated by the elementary braids
$\sigma_{1}, \sigma_{2}, \dots ,\sigma_{n-1}$ subject to the
following relations:
\begin{align*}
\sigma_{i} \sigma_{j} & =\sigma_{j} \sigma_{i} \mbox{ if } |i-j| \geq 2\\
\sigma_{i} \sigma_{i+1} \sigma_{i} & = \sigma_{i+1}
\sigma_{i}\sigma_{i+1}, \ \forall \ 1 \leq i \leq n-2.
\end{align*}

By Alexander's theorem  \cite{Al1}, any   link $L$ in $S^3$
can be represented  as the closure of an $n$-braid $b$. We write  $L=\hat{ b}$. This representation is not unique and the minimum integer $n$ such that $L$ admits an $n$-braid representation
 is called the braid index of $L$. Recall that two braids represent the same link if and only if they are related by a finite sequence  of  Markov moves \cite{Mu1985}.
The group $B_2$ is infinite cyclic. Consequently, a  link of braid index two is a  $(2,k)$-torus link.  A complete   classification of 3-braids up to conjugacy  can be found in  \cite{Mu1974}. This classification has been  very useful in the study of links of braid index 3. In particular,   Stoimenow  \cite{St} determined which 3-braid links are alternating. Baldwin \cite{Ba} characterized quasi-alternating
 links of braid index 3. More recently, Lee and Vafaee \cite{LV}  proved  that twisted torus knots are the only $L$-space knots of braid index 3. As for knot invariants,   signatures  of  3-braid links are computed  in \cite{Er} and explicit formulas for their determinants are given in \cite{QC}.
 The structure of the Jones polynomials of 3-braid links is studied in \cite{CQ,Ch2022}. Colored HOMFLY-PT invariants of some families of 3-braid links have been computed in \cite{SMR,SC}. A list of open questions related to the study of 3-braids is given by Birman and Menasco  \cite{BM}.\\

The  Alexander polynomial \cite{Al2} is a topological   invariant of oriented links which associates with each link $L$ a Laurent polynomial with  integral coefficients
  $\Delta_L(t) \in \mathbb{Z}[t^{\pm 1/2}]$. Conway proved that this
polynomial can be defined in  a simple recursive way. This polynomial can be also recovered from the Burau representation of the braid group \cite{Bur}.  More recently, Ozsv$\acute{a}$th and  Szab$\acute{o}$ \cite{OS}
 proved that this polynomial, up to the multiplication by some factor,  is   obtained as the Euler characteristic of link Floer homology. The Alexander polynomial of a knot $K$ is known to be  symmetric; it satisfies   $\Delta_K(t)=\Delta_K(t^{-1})$. Furthermore,  up to  the multiplication by a power of $t$  it is always  possible to write  $\Delta_K(t)=\displaystyle\sum_{i=0}^{2n}\alpha_it^i$ with  $\alpha_{0}\neq 0$ and  $\alpha_{2n}\neq 0$.  \\
Murasugi, studied the Alexander polynomial of alternating knots and proved that the coefficients of this polynomial alternate in sign and that the polynomial has no gaps, equivalently $\alpha_i\alpha_{i+1}<0$ for all $0\leq i < 2n$. Moreover, the degree of $\Delta_K(t)$ is equal to twice  the genus of the knot \cite{Mu1958a,Mu1958b}. In \cite{Fo}, R. Fox conjectured that the coefficients of the Alexander polynomial of an alternating knot are  trapezoidal. Stoimenow, later conjectured that these coefficients are indeed log-concave \cite{St2005}.

\begin{conj}\label{conj}\cite{Fo}
Let $K$ be an alternating knot and  $\Delta_K(t)=\pm \displaystyle\sum_{i=0}^{2n}\alpha_i(-t)^i$, with $\alpha_i>0$,   its  Alexander polynomial. Then there exists an integer $r \leq n$ such that:
$$\alpha_0<\alpha_1< \dots < \alpha_{n-r}=\dots =\alpha_{n+r}> \dots ...>\alpha_{2n-1}>\alpha_{2n}.$$
\end{conj}

This conjecture, known as    Fox's trapezoidal conjecture, was listed by J. Huh \cite{Hu}, as one of the most interesting open problems about log-concavity. The conjecture has been confirmed  for several classes of alternating knots.  Hartley  proved the conjecture for 2-bridge knots  \cite{Ha}. Murasugi, showed that  the conjecture holds  for
a large family of  alternating algebraic knots \cite{Mu1985}. Using knot Floer homology,  Ozsv$\acute{a}$th and  Szab$\acute{o}$ \cite{OS}  confirmed  the conjecture for knots of genus 2. The same result has been obtained by Jong using  combinatorial methods \cite{Jo1,Jo2}. Hirasawa and Murasugi \cite{HM} showed  that Conjecture \ref{conj} holds  for stable alternating knots and suggested a refinement of this conjecture by stating that the length of the stable part  is less than the signature of the knot $\sigma(K)$, which means $r\leq |\sigma(K)|/2$.  In  \cite{Ch}, Chen  proved   that this refinement holds for  two-bridge knots. Alrefai and the second author, showed that certain families of 3-braid knots satisfy Fox trapezoidal conjecture in its refined form  \cite{AC}.\\
Let $m \geq 1$ be an integer. The closure of the 3-braid $(\sigma_{1}^{m}\sigma_{2}^{-m})^n$ is called the  generalized weaving link $W(3,n,m)$, see \cite{SMR,SC}. Notice that if $m=1$, then we get the classical weaving link often denoted as $W(3,n)$.  In this paper, we  establish a relationship between the trace of the Burau representation of the braids of type
 $(\sigma_{1}^{m}\sigma_{2}^{-m})^n$ and the Chebyshev polynomials of the first  kind. As an application, we introduce explicit formulas for the coefficients of  Jones and Alexander polynomials of weaving links. Furthermore, we verify that the Alexander polynomial of weaving knots satisfy Fox's trapezoidal conjecture,  see Figure \ref{figure0}.
We also compute the zeroes of the Alexander polynomials of weaving knots and we verify that the real  part of any such zero is  greater than $-1$, hence these zeroes   satisfy  Hoste's condition, see \cite{LMu} and \cite{HIS}. By setting $s=-t$, we have:
\begin{thm}\label{main}
\begin{enumerate}
\item The Alexander polynomial of the weaving link $W(3,n)$ is given by the following formula:
 $\Delta_{W(3,n)} (s) = \sum_{k=0}^{2n-2} \alpha_{n,k} s^{k},$ with:
		$$ \alpha_{n,k} = \sum_{i=0}^{n-1} (-1)^{k-n+i+1} \frac{2n(n+i)!}{(n-i-1)! (2i+2)!}  \begin{pmatrix}
			i;3 \\
			k+i-n+1
		\end{pmatrix} $$
		where $\begin{pmatrix}
			n;3 \\ k
		\end{pmatrix}$ is the coefficient of $q^k$ in the expansion of $(1+q+q^2)^n.$

\item For any $n\geq 1$, the coefficients of $\Delta_{W(3,n)}(s)$ are trapezoidal.
\end{enumerate}
\end{thm}

\begin{figure}[h]
$$\begin{array}{ccccccccccccccccc}
\bf 1&&&&&&&&&&&&&&&&\\
1&\bf 3&1&&&&&&&&&&&&&&\\
1&4&\bf 6&4&1&&&&&&&&&&&&\\
1& 5& 10& \bf13& 10& 5& 1&&&&&&&&&&\\
1& 6& 15& 24& \bf 29& 24& 15& 6& 1&&&&&&&&\\
1 &7& 21& 40& 58&\bf  66& 58& 40& 21& 7& 1&&&&&&\\
1 &8 &28& 62& 104 &140 &\bf155 &140 &104& 62& 28& 8& 1&&&&\\
1& 9& 36& 91& 173& 266& 341&\bf 371& 341& 266& 173& 91& 36& 9& 1&&\\
1 &10& 45& 128& 272& 468& 676& 838& \bf 900& 838& 676& 468& 272& 128& 45& 10& 1
\end{array}
$$
\caption{The coefficients of the Alexander polynomials of $\Delta_{W(3,n)} (s)$, for $ n=1$ to 10. }
\label{figure0}
\end{figure}

Here is an outline of this paper. In Section 2, we briefly recall some basic definitions and notations needed in  the sequel. In Section 3, we study the Jones polynomial of generalized weaving links. Section 4 is devoted to prove our main theorem. Finally, the zeroes of the Alexander polynomial of weaving knots are computed in Section 5.

\parindent 0cm
	\section{Preliminaries}

In 1928, Alexander \cite{Al2} introduced a  topological invariant of oriented links. This invariant  is a one-variable Laurent polynomial with integral coefficients  $\Delta_L(t)$.  It can be defined in several different, but equivalent, ways. In particular, it  is uniquely determined   recursively   using Conway skein relations:
$$\begin{array}{l}
\Delta_U(t)=1,\\
\Delta_{L_+}(t) - \Delta_{L_-}(t)= (\sqrt{t}-\displaystyle\frac{1}{\sqrt{t}}) \Delta_{L_0}(t),
\end{array}$$
where $U$ is the unknot and  $L_{+}$, $L_{-}$ and  $L_{0}$ are three oriented link diagrams
which are identical except in a small region  where they are as pictured in Figure \ref{figure1}.
\begin{figure}[h]
\centering
\includegraphics[width=10cm,height=2.5cm]{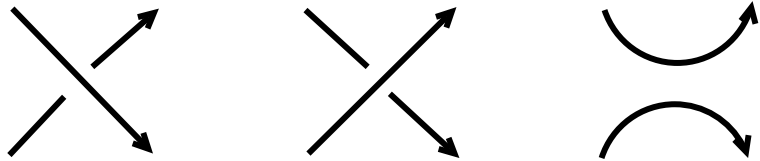}
\caption{ The three oriented links  $L_+$, $L_{-}$ and  $L_0$,  respectively.}
\label{figure1}
\end{figure}

 This polynomial is known to be symmetric, in the sense that it satisfies  $\Delta_L(t)= \Delta_L(t^{-1})$. The reduced Burau representation of the braid group permits  another way to define the Alexander polynomial. Since we are only interested in 3-braid links, we shall restrict ourselves to the Burau representation of    $B_3$. This representation is faithful and can be described  as follows. Let   $b$ be a given 3-braid and  $e_b$  the exponent sum of $b$ as a word in the
elementary braids $\sigma_1$ and $\sigma_2$.   Let $\psi_t: B_{3}
\longrightarrow GL(2, \mathbb{Z}[t,t^{-1}])$ be the Burau
representation defined on the generators of $B_3$ by $\psi_t(\sigma_1)=\left [\begin{array}{cc}
-t&1\\
0&1
\end{array} \right ] $ and $\psi_t(\sigma_2)=\left [\begin{array}{cc}
1&0\\
t&-t
\end{array} \right ] $. The Alexander polynomial, up to the multiplication by a unit of  $\mathbb{Z}[t,t^{-1}]$,  of the link $L=\hat b$ is given by:
$$\Delta_L(t)=(-1/\sqrt{t})^{e_b-2}\frac{1-t}{1-t^3}\mbox{det}(I_2-\psi_t(b)),$$
where $I_2$ is the identity matrix.
 Given that this  representation is 2-dimensional, then $\mbox{det}(I_2-\psi_t(b))=1-\mbox{tr}(\psi_t(b))+\mbox{det}(\psi_t(b))$, where  $\mbox{tr}$ denotes the usual matrix-trace function defined as the sum of  entries on the main diagonal of  the matrix.\\

In 1984, V. F. Jones introduced a new polynomial invariant of oriented links. This one-variable Laurent polynomial, denoted hereafter by $V_L(t)$, was originally defined  through the study of  representations of the braid group via  certain  von Neumann algebras \cite{Jo}. The Jones polynomial  is uniquely determined by the following skein  relations:\\
 $$\begin{array}{l}
V_U(t)=1,\\
tV_{L_+}(t) - t^{-1} V_{L_-}(t)= (\sqrt{t}+\displaystyle\frac{1}{\sqrt{t}}) V_{L_0}(t).
\end{array}$$

In the case of 3-braid links, Birman \cite{Bi} showed  that the Jones polynomial   can be defined using the   Burau representation of $B_3$ by the following formula:
 $$
 V_{\hat\alpha}(t)=(-\sqrt{t})^{e_\alpha}(t+t^{-1}+\mbox{tr}(\psi_t(\alpha))).$$

Throughout the rest of this paper, we  set $s=-t$. Moreover,  for any positive integer $n$ we set $[n]=\frac{1-s^n}{1-s}$. Let $T_n(x)$  be the Chebyshev polynomial of the first kind.  This  polynomial is determined recursiveley by the following  relations:
		\begin{align*}\label{chebRec}
			T_0(x) &= 1\\
			T_1(x) &= x\\
			T_{n+1}(x) &= 2xT_n(x) - T_{n-1}(x).
		\end{align*}
Let $F_n(q)=\sum_{k=0}^{n} f_{n,k} q^k $ and   $C_n(q)= \sum_{k=0}^{2n} c_{n,k} q^k $ be the rank polynomials of the Fibonacci lattice and the Lucas lattice, respectively.
Explicit formulas for the Whitney numbers $f_{n,k}$ and  $c_{n,k}$ have been obtained
 by  Munarini and  Zagaglia Salvi \cite{MZ}. Another formula  for $c_{n,k}$ is given in  \cite{OEIS}:\\
$$c_{n,k} = \begin{cases}
		\sum_{j=0}^{\lfloor k/2 \rfloor} \frac{n}{n-j}\binom{n-j}{n-k+j}\binom{k-j-1}{j} & k < 2n \\
		1 & k = 2n.
	\end{cases}$$
The relation between the  rank polynomial of Lucas lattice and the Chebyshev polynomial is given by the following proposition \cite{MZ}.
	\begin{prop}\label{Prop_chebyshev}
		For any $n\geq 0$, we have: $C_n(q) = 2q^nT_n(\frac{1+q+q^2}{2q})$.
	\end{prop}

\section{Burau representation of 3-braids }
	In this section, we shall study   the Burau representation of the braids $b_n=(\sigma_1\sigma_2^{-1})^n$  as well as  $b_{n,m}=(\sigma_1^m\sigma_2^{-m})^n$, where $n$ and $m$ are positive integers. Recall that these  braids close to the weaving link  $W(3,n)$ and the generalized weaving link $W(3,n,m)$, respectively. In particular, we will give an explicit formula for the coefficients of
Jones polynomials of weaving links.
	\begin{prop}{\label{Prop_Trace}}
		The trace of the  Burau representation of the generalized weaving braid  is given by:
		$${ \rm{tr}}(\psi_s(b_{n,m}))= 2^{-n} s^{-nm} (\lambda_{1,m}^n + \lambda_{2,m}^n)$$
		where $$
		\begin{array}{rl}
			\lambda_{1,m} &= (1 +  [m]^2 s + s^{2m}) + \sqrt{(1 + [m]^2 s + s^{2m})^2  - 4s^{2m}}\;\;  {\rm{ and }}\\
			\lambda_{2,m} &= (1 +  [m]^2 s + s^{2m}) -  \sqrt{(1 + [m]^2 s + s^{2m})^2  - 4s^{2m}}.
		\end{array}  $$
	\end{prop}
	
	\begin{proof}
		
		Notice first that one can easily show that  for any $m>0$, we have:
		\begin{center} $\psi_s(\sigma_1^{m})=\left [\begin{array}{cc}
				s^{m}& [m]\\
				&\\
				0&1
			\end{array} \right ] $, and
			$\psi_s(\sigma_2^{-m})= s^{-m}\left [\begin{array}{cc}
				s^{m}&0\\
				&\\
				s[m]&1
			\end{array} \right ] $.
		\end{center}
		
		Thus,
		
		$$\psi_s(\sigma_1^{m}\sigma_2^{-m})=
		s^{-m}\left [\begin{array}{cc}
			s^{2m}+s[m]^2& [m]\\
			&\\
			s[m]& 1
		\end{array} \right ].$$		
		
		Now, assume that  $\psi_s(b_{n,m}) = s^{-nm} \begin{bmatrix}
			A_n(s) & B_n(s) \\
			C_n(s) & D_n(s)
		\end{bmatrix}$, where all entries are polynomials on the variable $s$. Since $b_{n+1,m} = b_{n,m} \sigma_1^m \sigma_2^{-m}$, we get the following system of recurrence relations:
		\begin{align*}
			A_{n+1}(s) & = (s^{2m} +s [m]^2) A_n(s)  + s[m] B_n(s) \\
			B_{n+1}(s) & = [m]A_n(s)  + B_n(s) \\
			C_{n+1}(s) & = (s^{2m} +s [m]^2) C_n(s)  + s[m] D_n(s) \\
			D_{n+1}(s) & = [m]C_n(s)  + D_n(s)
		\end{align*}
		with $A_0(s) = D_0(s) = 1$ and $B_0(s) = C_0(s) = 0$ .  Hence:
		\begin{align*}
			\begin{bmatrix}
				A_{n+1}(s) \\
				B_{n+1}(s) \\
			\end{bmatrix} &= \begin{bmatrix}
				(s^{2m} +s [m]^2) & s[m] \\
				[m] & 1
			\end{bmatrix} \begin{bmatrix}
				A_{n}(s) \\
				B_{n}(s) \\
			\end{bmatrix}\\
			&= \begin{bmatrix}
				(s^{2m} +s [m]^2) & s[m] \\
				[m] & 1
			\end{bmatrix}^{n+1} \begin{bmatrix}
				A_{0}(s) \\
				B_{0}(s) \\
			\end{bmatrix} \\
			&= \begin{bmatrix}
				(s^{2m} +s [m]^2) & s[m] \\
				[m] & 1
			\end{bmatrix}^{n+1} \begin{bmatrix}
				1 \\
				0 \\
			\end{bmatrix}.
		\end{align*}
		
		Similarly, we also get for $C_n(s)$ and $ D_n(s)$:
		
		\begin{align*}
			\begin{bmatrix}
				C_{n+1}(s) \\
				D_{n+1}(s) \\
			\end{bmatrix}
			&= \begin{bmatrix}
				(s^{2m} +s [m]^2) & s[m] \\
				[m] & 1
			\end{bmatrix}^{n+1} \begin{bmatrix}
				0 \\
				1 \\
			\end{bmatrix}.
		\end{align*}
		
		Solving for the eigenvalues of $\begin{bmatrix} (s^{2m} +s [m]^2) & s[m] \\
			[m] & 1 \end{bmatrix}^n$ then diagonalizing the matrix gives us:
		\begin{align*}
			A_n(s) + D_n(s) &= 2^{-n} (\lambda_{1,m}^n + \lambda_{2,m}^n). \mbox{Thus,} \\
		 \mbox{tr}(\psi_s(b_{n,m})) &= s^{-n} (A_n(s) + D_n(s))\\
			&= 2^{-n} s^{-n} (\lambda_{1,m}^n + \lambda_{2,m}^n).
		\end{align*}
	\end{proof}
	
	\begin{coro}
		We have: ${\rm{tr}}(b_n)=2^{-n} s^{-n} (\lambda_1^n + \lambda_2^n)$, where
		
		$$\begin{array}{rl}
			\lambda_{1} & = (1 +   s + s^{2})+ \sqrt{(1 +  s + s^{2})^2  - 4s^{2}}\;\;  {\rm{ and }}\\
			
			\lambda_{2}&= (1 +   s + s^{2}) -  \sqrt{(1 +  s + s^{2})^2  - 4s^{2}}.
		\end{array} $$
	\end{coro}
	
	\newpage
	
	\begin{thm}\label{Prop_Trace_Cheby}
		The Jones polynomial of the generalized weaving knot $W(3,n,m)$ is given by:
		$$ V_{W(3,n,m)}(s) = -s-s^{-1}+2T_n(\frac{1 +  [m]^2 s + s^{2m}}{2s^m}).$$
	\end{thm}
	
	\begin{proof}
		The Jones polynomial can be obtained through the Burau representation:
		$$V_{W(3,n,m)}(s) = -s -s^{-1}+{\rm{tr}}(\psi_s(\sigma_1^m \sigma_2 ^{-m})^n).$$
		
		It is known that $T_n(x) = \frac{1}{2}((x - \sqrt{x^2 -1})^n + (x + \sqrt{x^2 -1})^n)$, hence it becomes clear that:
		$$T_n(\frac{1 +  [m]^2 s + s^{2m}}{2s^m}) = 2^{-n-1} s^{-nm} (\lambda_1^n + \lambda_2^n) = \frac{{\rm{tr}}(\psi_s(\sigma_1^m \sigma_2 ^{-m})^n)}{2}.$$
		Consequently, $$V_{W(3,n,m)}(s) = -s -s^{-1}+2T_n(\frac{1 +  [m]^2 s + s^{2m}}{2s^m}).$$
	\end{proof}
	
	\begin{coro}\label{WeaveJones} The Jones polynomial of the weaving knot $W(3,n)$ is given by:
		$$ V_{W(3,n)}(s) = -s-s^{-1}+s^{-n}C_n(s) =\sum_{k=0}^{2n} a_k s^{k-n},$$ where
		$$a_k = \begin{cases}
			c_{n,k} & |k-n| \neq 1\\
			c_{n,k}-1 & |k-n| = 1
		\end{cases}$$
	\end{coro}
	
	\begin{proof}
		By Theorem  \ref{Prop_Trace_Cheby}, we know that ${\rm{tr}}(\psi_s( (\sigma_1 \sigma_2 ^{-1})^n )) = -s-s^{-1}+2T_n(\frac{1 + s + s^{2}}{2s})$, and by Proposition \ref{Prop_chebyshev} we have that:
		$2T_n(\frac{1 + s + s^{2}}{2s}) =s^{-n}C_n(s) =\sum_{k=0}^{2n} c_{n,k} s^{k-n}.$
	\end{proof}
	
	The determinant of an oriented link $L$ is a classical numerical invariant of links, $\det(L)$, which can be obtained by evaluation of  the Jones polynomial at $t=-1$;  $\det(L)=|V_L(-1)|$.  The  Lucas  numbers $L_{k}$ are defined  recursively by the  relations: $L_0=2, L_1=1$ and $L_{k}=L_{k-1}+L_{k-2}$ for any $k\geq 2$. The generalized Lucas numbers $L_{m,k}$ are defined by the relation
	$L_0=2, L_1=1$ and $L_{m,k}=mL_{m,k-1}+L_{m,k-2}$.  The relationship between the determinants of weaving links and the Lucas numbers has been observed first in \cite{SC}.
	Then this relation was extended to the wider class of generalized weaving links  in \cite{JP}. These results can be  obtained as a direct consequence of Theorem \ref{Prop_Trace_Cheby} above.\\
	
	\begin{coro}\label{lucasnumbers} For any $n\geq 1$ and $m \geq 1$, we have:
		$$\begin{array}{ll}
			\det(W(3,n)) &=L_{2n}-2 \\
			\det(W(3,n,m)) &=L_{m,2n}-2.
		\end{array}
		$$
	\end{coro}
	
	\begin{proof}
		Using  the recurrence relation for $L_{m,k}$ we obtain $$L_{m,k} = \left(\frac{m+\sqrt{m^2 +4}}{2}\right)^k + \left(\frac{m-\sqrt{m^2 +4}}{2}\right)^k.$$
		
		Then notice that:
		\begin{align*}
			\det(W(3,n,m)) 	&= |V_{W(3,n,m)}(1)|\\
			&= -2 + 2T_n(\frac{2 +  m^2}{2}) \\
			&= -2 + \left(\frac{2+m^2}{2} + \sqrt{\left(\frac{2+m^2}{2}\right)^2 -1}\right)^n + \left(\frac{2+m^2}{2} - \sqrt{\left(\frac{2+m^2}{2}\right)^2 -1}\right)^n\\
			&= -2 + \left(\frac{m+\sqrt{m^2 +4}}{2}\right)^{2n} + \left(\frac{m-\sqrt{m^2 +4}}{2}\right)^{2n}\\
			&= -2 + L_{m,2n}.
		\end{align*}
Similarly, $\det(W(3,n))=L_{n} -2$.
	\end{proof}
	
	We close this section by the following proposition which will be used later to  prove that the  Alexander polynomial of weaving  knots  is trapezoidal.
	
	\begin{prop}\label{recurrence}
		The Whitney numbers of the Lucas lattice $c_{n,k}$ satisfy the following recurrence relation:
		\begin{equation}
			\begin{aligned}
				c_{n,0} &= 1\\
				c_{n,1} &= n = c_{n-1,1} + c_{n-1,0}\\
				c_{n,k} &= c_{n-1,k}+c_{n-1,k-1}+c_{n-1,k-2}-c_{n-2,k-2}.\\
			\end{aligned}
		\end{equation}
	\end{prop}
	
	\begin{proof}
		For $k=0$ and $k=1$ this follows  straightforward from the explicit formula for $c_{n,k}$.
		
		For $k\geq 0$, we have   the following identity relating the sequences  $c_{n,k}$  and $f_{n,k}$,  \cite{MZ}:
		\begin{equation}\label{cnk}
			c_{n+2, k+2} = f_{2n+4, k+2} + f_{2n, k},
		\end{equation}
		where \begin{equation}\label{fnk}
			f_{n+4,k+2} = f_{n+2,k+2} + f_{n+2,k+1} + f_{n+2,k} - f_{n,k}.
		\end{equation}
		
		Substituing  \ref{fnk} in \ref{cnk} we get:
		\begin{align*}
			c_{n+2, k+2} &= f_{2n+2,k+2} + f_{2n+2,k+1} + f_{2n+2,k} - f_{2n,k} + f_{2n-2,k}
			\\&\phantom{=} + f_{2n-2,k-1} + f_{2n-2,k-2} - f_{2n-4,k-2}\\
			&= c_{n+1, k+2} + c_{n+1, k+1} + c_{n+1, k} - c_{n, k}.
		\end{align*}
		
		Giving us $c_{n,k} = c_{n-1,k}+c_{n-1,k-1}+c_{n-1,k-2}-c_{n-2,k-2}$, for $k\geq 2$.
	\end{proof}
	\vspace{5 pt}
	\section{Proof of the Main Theorem}
This section is devoted to the proof of Theorem \ref{main}. First,  we shall give explicit formulas for the coefficients of the Alexander polynomial of weaving links $W(3,n)$.  More precisely, we will prove that:
$\Delta_{W(3,n)} (s) = (\sum_{k=0}^{2n-2} \alpha_{n,k} s^{k})$ where
$$ \alpha_{n,k} = \sum_{i=0}^{n-1} (-1)^{k-n+i} \frac{2n(n+i)!}{(n-i-1)! (2i+2)!}  \begin{pmatrix}
	i;3 \\
	k+i-n+1
\end{pmatrix}. $$

We know that $T_n(x) = n\sum_{k=0}^n (-2)^k \frac{(n+k-1)!}{(n-k)! (2k)!} (1-x)^k$, then:

\begin{align*}
	2-2T_n(\frac{1+s+s^2}{2s}) 	&= 2 - 2n\sum_{k=0}^n (-2)^k \frac{(n+k-1)!}{(n-k)! (2k)!} (1-\frac{1+s+s^2}{2s})^k \\
	&= 2 - 2 n\sum_{k=0}^n (2)^k \frac{(n+k-1)!}{(n-k)! (2k)!} (\frac{1-s+s^2}{2s})^k \\
	&=  - 2 n\sum_{k = 1}^n (2)^k \frac{(n+k-1)!}{(n-k)! (2k)!} (\frac{1-s+s^2}{2s})^k \\
	&= -2(\frac{1-s+s^2}{2s}) n\sum_{k = 1}^n (2)^k \frac{(n+k-1)!}{(n-k)! (2k)!} (\frac{1-s+s^2}{2s})^{k-1}\\
	&= \frac{-(1-s+s^2)}{s} n\sum_{k = 1}^n (2)^k \frac{(n+k-1)!}{(n-k)! (2k)!} (\frac{1-s+s^2}{2s})^{k-1}\\
	&= \frac{-(1-s+s^2)}{s} n\sum_{k = 0}^{n-1} (2)^{k+1} \frac{(n+k)!}{(n-k-1)! (2k+2)!} (\frac{1-s+s^2}{2s})^{k}\\
	&= \frac{-(1-s+s^2)}{s} n\sum_{k = 0}^{n-1}   \frac{2(n+k)!}{(n-k-1)! (2k+2)!} \sum_{l=0}^{2k} (-1)^{l} \begin{pmatrix}
		k;3 \\
		l
	\end{pmatrix} s^{l-k}\\
	&= \frac{-(1-s+s^2)}{s} n\sum_{k = 0}^{n-1}   \frac{2(n+k)!}{(n-k-1)! (2k+2)!} \sum_{l=-k}^{k} (-1)^{l+k} \begin{pmatrix}
		k;3 \\
		l+k
	\end{pmatrix} s^{l}\\
	&= -\frac{1-s+s^2}{s} \sum_{l = -n+1}^{n-1}\ \sum_{k=0}^{n-1} (-1)^{l+k} \frac{2n(n+k)!}{(n-k-1)! (2k+2)!}  \begin{pmatrix}
		k;3 \\
		l
	\end{pmatrix} s^{l}\\
	&= -\frac{1-s+s^2}{s} \sum_{l = -n+1}^{n-1}\ \sum_{k=0}^{n-1} (-1)^{l+k} \frac{2n(n+k)!}{(n-k-1)! (2k+2)!}  \begin{pmatrix}
		k;3 \\
		l+k
	\end{pmatrix} s^{l}\\
	&= -\frac{1-s+s^2}{s} \sum_{l = -n+1}^{n-1}\ \sum_{k=0}^{n-1} (-1)^{l+k} \frac{2n(n+k)!}{(n-k-1)! (2k+2)!}  \begin{pmatrix}
		k;3 \\
		l+k
	\end{pmatrix} s^{l}.
\end{align*}

Therefore, up to some multiplication of a unit in $\mathbb{Z}[s,s^{-1}]$, we have:  $$\Delta_{W(3,n)} (s)  = \frac{2-2Tn(\frac{1+s+s^2}{2s})}{1-s+s^2} = - \sum_{l = -n+1}^{n-1}\ \sum_{k=0}^{n-1} (-1)^{l+k} \frac{2n(n+k)!}{(n-k-1)! (2k+2)!}  \begin{pmatrix}
	k;3 \\
	l+k
\end{pmatrix} s^{l}.$$
We could then shift the polynomial by $s^{n-1}$ to give us
$$\Delta_{W(3,n)} (s)  =  \sum_{l = 0}^{2n-2}\ \sum_{k=0}^{n-1} (-1)^{l-n+k+1} \frac{2n(n+k)!}{(n-k-1)! (2k+2)!}  \begin{pmatrix}
	k;3 \\
	l+k-n+1
\end{pmatrix} s^{l}.$$

Now, in order to prove the second part of  Theorem \ref{main} we  need the following Lemma.

\begin{lemm}
	Let $\Delta_{W(3,n)} (s) = (\sum_{k=0}^{2n-2} \alpha_{n,k} s^{k})$, then $\alpha_{k,n}$ satisfies the following recurrence relations:
	\begin{equation*}
		\begin{aligned}
			\alpha_{n,0} &= 1\\
			\alpha_{n,1} &= \alpha_{n-1,1} + \alpha_{n-1,0}\\
			\alpha_{n,k} &= \alpha_{n-1,k} + \alpha_{n-1,k-1} + \alpha_{n-1,k-2} - \alpha_{n-2,k-2} \text{ \ for $1<k<n-1$}\\
			\alpha_{n,n-1} &= \alpha_{n-1,n-1} + \alpha_{n-1,n-2} + \alpha_{n-1,n-3} - \alpha_{n-2,n-3} +2.
		\end{aligned}
	\end{equation*}
\end{lemm}

\begin{proof}
	Recall that the Burau representation defines the  Alexander polynomial up to
the multiplication by a unit. In the following calculations,
 we shall use the symbol   $ f \equiv g $  to indicate that   the
polynomials $f$ and $g \in \mathbb{Z}[s,s^{-1}]$ satisfy  $f = u . g$ for some unit $u \in \mathbb{Z}[s,s^{-1}]$.
	
	Since $\Delta_{W(3,n)} (s) \equiv \frac{1+s}{1+s^3} \det(I-\psi_s((\sigma_1 \sigma_2 ^{-1})^n))$ and $\det(I-\psi_s((\sigma_1 \sigma_2 ^{-1})^n))$ is the characteristic polynomial evaluated at $\lambda=1$ of the $2 \times 2$ matrix $\psi_s((\sigma_1 \sigma_2 ^{-1})^n)$ then:
	$$\Delta_{W(3,n)} (s) \equiv \frac{1+s}{1+s^3} (1- {\rm{tr}}(\psi_s( (\sigma_1 \sigma_2 ^{-1})^n ))+ \det(\psi_s( (\sigma_1 \sigma_2 ^{-1})^n ))).$$
	
	Notice that $\det(\psi_s( (\sigma_1 \sigma_2 ^{-1})^n )) = \det(s^{-1} \begin{bmatrix}
		s^2+s & 1\\
		s & 1
	\end{bmatrix})^n = 1.$
	
	Therefore:
	\begin{align*}
		\Delta_{W(3,n)} (s) &\equiv \frac{1+s}{1+s^3} (2- {\rm{tr}}(\psi_s( (\sigma_1 \sigma_2 ^{-1})^n )))\\
		&\equiv \frac{1+s}{1+s^3} (2s^n- \sum_{k=0}^{2n} c_{n,k} s^{k})\\
		&\equiv (1+s)(2s^n- \sum_{k=0}^{2n} c_{n,k} s^{k})(\sum_{i=0} (-s)^{3i}).
	\end{align*}
	
	Now if we define $c^*_{n,k}$ to be:
	\begin{align*}
		c^*_{n,0} &= c_{n,0} = 1\\
		c^*_{n,k} &= c_{n,k}+c_{n,k-1} \text{ \ for $0<k<n$ }\\
		c^*_{n,n} &= c_{n,n} + c_{n,n-1} -2\\
		c^*_{n,k} &= c^*_{n,2n-k+1} \text{ \ for $k>n$ }
	\end{align*}
	and $\alpha_{n,k} = c^*_{n,k} - c^*_{n,k-3}+ c^*_{n,k-6} - ... = \sum_{i=0}^{k/3} (-1)^{i} c^*_{n,k-3i}$, then we get:
	\begin{align*}
		\Delta_{W(3,n)} (s) &\equiv ( \sum_{k=0}^{2n+1} c^*_{n,k} s^{k})(\sum_{i=0} (-s)^{3i})\\
		&\equiv  \sum_{k=0}^{2n-2} \alpha_{n,k} s^{k}.
	\end{align*}
	
	Notice that by Proposition  \ref{recurrence} we have:\\
	\noindent
	If  $k=1$, then
	\begin{align*}
		c^*_{n,1} 	&= c_{n,1}+c_{n,0}\\
		&= c_{n-1,1} + c_{n-1,0} + c_{n-1,0}\\
		&= c^*_{n-1,1} + c^*_{n-1,0}.
	\end{align*}
	If  $2\leq k < n-1$, then
	\begin{align*}
		c^*_{n,k} 	&= c_{n,k}+c_{n,k-1}\\
		&= c_{n-1,k}+c_{n-1,k-1}+c_{n-1,k-2}-c_{n-2,k-2}\\
		&\phantom{=} c_{n-1,k-1}+c_{n-1,k-2}+c_{n-1,k-3}-c_{n-2,k-3}\\
		&= c^*_{n-1,k} + c^*_{n-1,k-1} + c^*_{n-1,k-2} - c^*_{n-2,k-2}.
	\end{align*}
	If  $k=n-1$, then
	\begin{align*}
		c^*_{n,n-1} 	&= c_{n,n-1}+c_{n,n-2}\\
		&= c_{n-1,n-1}+c_{n-1,n-2}+c_{n-1,n-3}-c_{n-2,n-3}\\
		&\phantom{=} c_{n-1,n-2}+c_{n-1,n-3}+c_{n-1,n-4}-c_{n-2,n-4}\\
		&= c^*_{n-1,k} + c^*_{n-1,k-1} + c^*_{n-1,k-2} - c^*_{n-2,k-2}+2.
	\end{align*}
	
	Hence it follows from the definition of $\alpha_{n,k}$ and the recurrence relation above  that:
	\begin{equation}\label{arecur}
		\begin{aligned}
			\alpha_{n,0} &= 1\\
			\alpha_{n,1} &= \alpha_{n-1,1} + \alpha_{n-1,0}\\
			\alpha_{n,k} &= \alpha_{n-1,k} + \alpha_{n-1,k-1} + \alpha_{n-1,k-2} - \alpha_{n-2,k-2} \text{ \ for $1<k<n-1$}\\
			\alpha_{n,n-1} &= \alpha_{n-1,n-1} + \alpha_{n-1,n-2} + \alpha_{n-1,n-3} - \alpha_{n-2,n-3} +2.
		\end{aligned}
	\end{equation}
\end{proof}

Notice that this immediately gives us $\alpha_{n,0} < \alpha_{n,1} < \alpha_{n,2}$. For $k>2$,
we  use the above system to show that for $2<k<n-1$ we have: $$\alpha_{n,k} = \alpha_{n-1,k} + \alpha_{n-1,k-1} + \alpha_{n-2,k-3} + \alpha_{n-3,k-5} + ... = \alpha_{n-1,k} + \sum_{i=1}^{\frac{k-1}{2}} \alpha_{n-i,k-2i+1}$$
and for $k=n-1$ we have:
$$\alpha_{n,k} = 2+ \alpha_{n-1,k} + \sum_{i=1}^{\frac{k-1}{2}} \alpha_{n-i,k-2i+1}.$$

This directly comes from the recurrence relation \ref{arecur}. Since $\alpha_{n-1,k-2}$ is related to  $\alpha_{n-2,k-2}$
we can expand down to obtain the above relation. Notice that  $\Delta_{W(3,2)}(s) = s^{-1} + 3 + s$ is trapezoidal, thus  we can induct on $n$ to show that $\Delta_{W(3,n)}(s)$ is trapezoidal. Indeed,  if we assume that  $\Delta_{W(n-1,3)}(s)$ is trapezoidal we obtain:
\begin{align*}
	\alpha_{n,k} 	&\geq  \alpha_{n-1,k} + \sum_{i=1}^{\frac{k-1}{2}} \alpha_{n-i,k-2i+1}\\
	&> 	\alpha_{n-1,k-1} + \sum_{i=1}^{\frac{k-1}{2}} \alpha_{n-i,k-2i}\\
	& = \alpha_{n,k-1}.
\end{align*}

Hence completing our proof since $\alpha_{n,k-1} < \alpha_{n,k}$ for $k\leq n-1$. Moreover, $\Delta_{W(3,n)}(s)$ is a symmetric polynomial, thus  it has its center at $n-1$. Therefore,   it satisfies Fox's trapezoidal conjecture and the length of the satble part is $r = 0 $ which is less than or equal  to $\frac{|\sigma(W(3,n))|}{2}$.
	\parindent 1.3cm
\section{Zeroes of the Alexander polynomial of weaving knots}
In 2022, J.  Hoste conjectured  that   if $z$ is a complex zero of the Alexander polynomial of an alternating knot, then  $z$ satisfies the condition $\mathfrak{Re}(z) > -1$. Lyubich and Murasugi proved that this conjecture    holds for a large class of two-bridge  knots, see \cite{LMu}. Stoimenow \cite{St2016} proved it is  true for knots of  genus less than or equal 4.  Ishikawa \cite{Ish} proved that Hoste's conjecture  holds for any 2-bridge knot. A thorough discussions of the distribution  of the zeros of Alexander polynomial of alternating knots  can be found in  \cite{HM}. More recently, Hirasawa, Ishikawa and Suzuki \cite{HIS} found counter-examples to  Hoste's conjecture. It is worth mentioning here that  in all these counterexamples, the zeroes   have real parts which are  very close to -1. Thus, the existence of a lower bound
to the real part of these zeros  remains still an open question.  In this section, we shall compute the zeroes of $\Delta_{W(3,n)}(t)$ and prove that for every such zero  the real part is  greater than -1.\\
\begin{prop}
	The zeroes of $\Delta_{W(3,n)} (t)$ are of the form:
	$$z = -\frac{1}{2} \left( 2cos(\frac{2k}{n} \pi) -1 \pm \sqrt{ (2cos(\frac{2k}{n} \pi) -1) ^2 -4 } \right),$$
	with $1 \leq k \leq \lfloor n/2 \rfloor.$
\end{prop}

\begin{proof}
	The zeroes of $\Delta_{W(3,n)} (s)$ correspond to the zeroes of $2 - 2T_n(\frac{1+s+s^2}{2s})$ with the zeroes of $1-s+s^2$ ignored.		
	If $2 - 2T_n(\frac{1+s+s^2}{2s}) =0$ then $T_n(\frac{1+s+s^2}{2s}) = 1$. These values correspond to the peaks of $T_n (x)$ for $-1 \leq x \leq 1$ which are indeed given by
	$x_k = cos(\frac{2k}{n} \pi)$ for $0 \leq k \leq \lfloor n/2 \rfloor$.	
Hence, the values of the zeroes of the Alexander polynomial  are the solutions of   $\frac{1+s+s^2}{2s} = cos(\frac{2k}{n} \pi)$. Solving for $s$ and taking $t=-s$ gives us the desired  form.
\end{proof}

\begin{coro}
	If $z$ is a zero of $\Delta_{W(3,n)} (t)$, then $\mathfrak{Re}(z) > -1$.
	Moreover, if $z$ is a non-real root then $|z| =1$.
\end{coro}

\begin{proof}
	If $(2cos(\frac{2k}{n} \pi) -1) ^2 -4 \geq 0$, then we have $cos(\frac{2k}{n} \pi) \leq \frac{-1}{2}$. Since $cos(\frac{2k}{n} \pi)\geq -1$ we get that $(2cos(\frac{2k}{n} \pi) -1) ^2 -4 \leq 5$. If $(2cos(\frac{2k}{n} \pi) -1) ^2 -4 \geq 0$, then  $z$ is real and we have
	
	\begin{align*}
		&-3 -\sqrt{5} \leq 2cos(\frac{2k}{n} \pi) -1 \pm \sqrt{ (2cos(\frac{2k}{n} \pi) -1) ^2 -4 } \leq -2 + \sqrt{5}\\
		\implies &\frac{-3 -\sqrt{5}}{2} \leq -z \leq \frac{-2 + \sqrt{5}}{2}\\
		\implies &z > -1.
	\end{align*}
Otherwise, if $(2cos(\frac{2k}{n} \pi) -1) ^2 -4 < 0$ then $z$ is a non-real root  and we have \\ $$\mathfrak{Re}(z) = \frac{1}{2}\left( 1- 2cos(\frac{2k}{n} \pi) \right)$$ $$\mathfrak{Im}(z) = \pm\frac{1}{2} \sqrt{ 4 - (1- 2cos(\frac{2k}{n} \pi)) ^2 }.$$
One can easily see that $\mathfrak{Re}(z) > -1$ and  $|z| = 1$.
\end{proof}


\section*{Acknowledgments}
 This research was funded  by  United Arab Emirates University, UPAR grant
 $\# G00004167.$

\end{document}